\documentclass[letterpaper, 10pt, conference]{ieeeconf}      

\IEEEoverridecommandlockouts                              
\newcommand{\ud}{\,\mathrm{d}}

\usepackage{graphics} 
\usepackage{epsfig} 
\usepackage{mathptmx} 
\usepackage{times} 
\usepackage{amsmath} 
\usepackage{amssymb}  
\usepackage{tabularx}
\usepackage{epstopdf}
\usepackage{epic,color}
\usepackage{mathrsfs}

\newtheorem{definition}{Definition}
\newtheorem{theorem}{Theorem}
\newtheorem{lemma}{Lemma}
\newtheorem{remark}{Remark}
\newtheorem{proposition}{Proposition}
\newtheorem{corollary}{Corollary}

\newlength{\noteWidth}
\long\def\notes#1{\ifinner
             {\tiny #1}
             \else
              \marginpar{\parbox[t]{\noteWidth}{\raggedright\tiny #1}}
               \fi}

\def\notes#1{\typeout{#1 !!!}}  

\def\spm#1{\notes{\textcolor{blue}{SPM: #1}}}

\def\pgm#1{\notes{\textcolor{blue}{PGM: #1}}}

\def\ty#1{\notes{\textcolor{blue}{TY:#1}}}

\newcounter{rmnum}
\newenvironment{romannum}{\begin{list}{{\upshape (\roman{rmnum})}}{\usecounter{rmnum}
\setlength{\leftmargin}{14pt}
\setlength{\rightmargin}{8pt}
\setlength{\itemsep}{2pt}
\setlength{\itemindent}{-1pt}
}}{\end{list}}

\newcounter{anum}

\def\eqdef{\mathbin{:=}}
\def\Prob{{\sf P}}
\def\Inov{I}

\def\IEEEQEDclosed{\mbox{\rule[0pt]{1.3ex}{1.3ex}}}

\def\qed{\ifmmode\IEEEQEDclosed\else{\unskip\nobreak\hfil
\penalty50\hskip1em\null\nobreak\hfil\IEEEQEDclosed
\parfillskip=0pt\finalhyphendemerits=0\endgraf}\fi}

\def\ddt{\frac{\ud}{\ud t}}
\def\v{{\sf K}}

\def\Re{\mathbb{R}}
\def\E{{\sf E}}

\def\ind{\text{\rm\large 1}}
\def\varble{\,\cdot\,}

\def\Sec#1{Sec~\ref{#1}}

\def\smopot{{\cal G}}

\def\bfmath#1{{\mathchoice{\mbox{\boldmath$#1$}}%
{\mbox{\boldmath$#1$}}%
{\mbox{\boldmath$\scriptstyle#1$}}%
{\mbox{\boldmath$\scriptscriptstyle#1$}}}}

\def\bfInov{\bfmath{I}}
\def\Inov{I} 

\def\bfPhi{\bfmath{\Phi}}
\def\bfmX{\bfmath{X}}

\def\Expect{{\sf E}}

\def\clD{{\cal D}}

\def\clL{{\cal L}}

\def\clZ{{\cal Z}}

\def\w{{\sf \Omega}}

\def\hah{\hat{h}}

\title{\LARGE \bf
Multivariable Feedback Particle Filter}

\author{Tao Yang, Richard S. Laugesen, Prashant G. Mehta, Sean P. Meyn
\thanks{Financial support from NSF grants EECS-0925534 and
  CCF-08-30776, the Simons foundation grant 204296 (Laugesen), 
  and the AFOSR grant FA9550-09-1-0190 is gratefully acknowledged. 
}
\thanks{T.~Yang and P.~Mehta are with the Coordinated
  Science Laboratory and the Department of Mechanical Science and
  Engineering,
  R.~Laugesen is with the Math department, all
  at the University of Illinois at Urbana-Champaign; S.~Meyn is with
  the Department of Electrical and Computer Engineering at University
  of Florida at Gainesville.
{\tt\scriptsize taoyang1@illinois.edu; mehtapg@illinois.edu;
  laugesen@illinois.edu; meyn@ufl.edu}}
}

\begin{document}

\maketitle
\thispagestyle{empty}

%
%
%
%

\begin{abstract}

In recent work it is shown that importance sampling can be avoided in
the particle filter through an innovation structure inspired by
traditional nonlinear filtering combined with Mean-Field Game
formalisms~\cite{huacaimal07,YinMehMeyShan10}.  
The resulting \textit{feedback
  particle filter} (FPF) offers significant variance improvements;
in particular, the algorithm can be applied to systems that are not
stable.  The filter comes with an up-front computational cost to
obtain the filter gain.  This paper describes new representations and
algorithms to compute the gain in the general multivariable
setting.   The main contributions are,
\begin{romannum}
\item   
Theory surrounding the FPF is improved:  Consistency is established in
the multivariate setting, as well as well-posedness of the associated
PDE to obtain the filter gain.

\item
The gain can be expressed as the gradient of a function, which is
precisely the solution to Poisson's equation for a related MCMC
diffusion (the Smoluchowski equation).   This provides a bridge to
MCMC as well as to approximate optimal filtering approaches such as
TD-learning, which can in turn be used to approximate the gain.

\item
Motivated by a weak formulation of Poisson's equation, a Galerkin finite-element algorithm is proposed for approximation of the gain.  Its performance is illustrated in numerical experiments.

\end{romannum}


%
\end{abstract}

\def\fp{(\ref{eqn:Signal_Process}, \ref{eqn:Obs_Process})}

\section{Introduction}
\label{sec:intro}

In a recent work, we introduced a new feedback control-based
formulation of the particle filter for the nonlinear filtering problem~\cite{YangMehtaMeyn11ACC},\cite{YangMehtaMeyn11CDC}.
The resulting filter is referred to as the {\em feedback particle
  filter}.  In~\cite{YangMehtaMeyn11ACC},\cite{YangMehtaMeyn11CDC}, the filter was described for the scalar case,
where the signal and the observation processes are both real-valued.
The aim of this paper is to generalize the scalar results of our
earlier papers to the multivariable filtering problem:
\begin{subequations}
\begin{align}
\ud X_t &= a(X_t)\ud t + \ud B_t,
\label{eqn:Signal_Process}
\\
\ud Z_t &= h(X_t)\ud t + \ud W_t,
\label{eqn:Obs_Process}
\end{align}
\end{subequations}
where $X_t\in\Re^d$ is the state at time $t$, $Z_t \in\Re^m$ is the
observation process, $a(\varble)$, $h(\varble)$ are $C^1$
functions, and $\{B_t\}$, $\{W_t\}$ are mutually independent
Wiener processes of appropriate dimension.  The covariance matrix of
the observation noise $\{W_t\}$ is assumed to be positive definite.
The function $h$ is a column vector whose $j$-th coordinate is denoted as $h_j$ (i.e.,
$h=(h_1,h_2,\hdots,h_m)^T$).  For
notational ease, the process noise $\{B_t\}$ is assumed to be a
standard Wiener process. By scaling, we may assume without loss of generality that the covariance matrices associated with $\{B_t\}$, $\{W_t\}$ are identity matrices.
\spm{Note that I have made $R=I$}

The objective of the filtering problem is to estimate the
posterior distribution of $X_t$ given the history $\clZ_t :=
\sigma(Z_s:  s \le t)$. The posterior is denoted by $p^*$, so
that for any measurable set $A\subset \Re^d$,
\begin{equation}
\int_{x\in A} p^*(x,t)\, \ud x   = \Prob\{ X_t \in A\mid \clZ_t \}.
\nonumber
\end{equation}
The filter is infinite-dimensional since it
defines the evolution, in  the space of probability measures,
of $\{p^*(\varble ,t) : t\ge 0\}$.  If $a(\varble)$, $h(\varble)$ are linear functions, the solution is given by the
finite-dimensional Kalman filter.
The article~\cite{budchelee07} surveys numerical methods to approximate the nonlinear filter.
One approach described in this survey is particle filtering.

\spm{needed?? via sequential importance sampling,
where particles are generated according to their importance weight at every time step~
}

The particle filter is a simulation-based algorithm to approximate the
filtering task~\cite{DouFreGor01}. The key step is the
construction of $N$ stochastic processes $\{X^i_t : 1\le i \le N\}$:
The value $X^i_t \in \Re^d$ is the state for the $i^{\text{th}}$
particle at time $t$. For each time $t$, the empirical distribution
formed by, the ``particle population'' is used to approximate the
posterior distribution.  Recall that this is  defined for any measurable set $A\subset\Re^d$ by,
\begin{equation}
p^{(N)}(A,t) = \frac{1}{N}\sum_{i=1}^N \ind\{ X^i_t\in A\}.\nonumber
\end{equation}
A common approach in particle filtering is called {\em
sequential importance sampling}, where particles are generated
according to their importance weight at every time
step~\cite{baincrisan07,DouFreGor01}.

In our earlier papers~\cite{YangMehtaMeyn11ACC},\cite{YangMehtaMeyn11CDC}, an alternative feedback control-based
approach to the construction of a particle filter was introduced; see
also~\cite{CriXiong09,MitterNewton04,DaumHuang10,Coleman_Allerton_2011,peqagusingom11}
for related approaches.  The resulting particle filter, referred to as the feedback particle
filter, was described for the scalar filtering problem (where
$d=m=1$).  The main result of this paper is to describe the feedback
particle filter for the multivariable filtering
problem~\eqref{eqn:Signal_Process}-\eqref{eqn:Obs_Process}:

The particle filter is a controlled system.  The dynamics of the
$i^{\text{th}}$ particle have the following gain feedback form,
\begin{equation}
\begin{aligned}
\ud X^i_t = a& ( X^i_t) \ud t + \ud B^i_t
\\
& + \v(X^i_t,t) \ud \Inov^i_t +\w(X^i_t,t) \ud t, \end{aligned}
\label{eqn:particle_filter_nonlin_intro}
\end{equation}
where $\{B^i_t\}$ are mutually independent  standard Wiener processes,
$\bfInov^i$ is similar to the \textit{innovation process} that appears
in the nonlinear filter,
\begin{equation}
\ud \Inov^i_t \eqdef \ud Z_t - \frac{1}{2}    (h(X^i_t) + \hat{h}) \ud t,
\label{e:in_intro}
\end{equation}
where $\hat{h} := \E[h(X_t^i)|\mathcal{Z}_t]$. In a numerical
implementation, we approximate $\hat{h} \approx
\frac{1}{N}\sum_{i=1}^N h(X_t^i) =: \hat{h}^{(N)}$.

The gain function $\v$ is obtained as a solution to  an Euler-Lagrange boundary value problem (E-L BVP):  For $j = 1, 2,
\hdots, m$,  the function $\phi_j$ is  
a solution to the second-order differential equation,
\begin{equation}
\label{eqn:EL_phi_intro}
\begin{aligned}
\nabla \cdot (p(x,t) \nabla \phi_j(x,t) ) & = - (h_j(x)-\hat{h}_j) p(x,t),\\
\int \phi_j(x,t) p(x,t) \ud x & = 0,
\end{aligned}
\end{equation}
where $p$ denotes the conditional distribution of $X_t^i$ given
$\mathcal{Z}_t$.   In terms of these solutions, the gain function is given by
\begin{equation}
[\v]_{lj} = 
\frac{\partial \phi_j} {\partial x_l} \, .
\label{eqn:gradient_gain_fn_intro}
\end{equation}
Note that the gain function needs to be obtained for each value of time $t$.

Finally, $\w = \left(\w_1,\w_2,...,\w_d\right)^T$ is the Wong-Zakai correction term:
\begin{equation}
\w_l(x,t) := \frac{1}{2} \sum_{k=1}^d \sum_{s=1}^m
								\v_{ks}(x,t)  \frac{\partial \v_{ls}}{\partial x_{k}}(x,t).
\label{eqn:wong_term_intro}
\end{equation}
The controlled system~\eqref{eqn:particle_filter_nonlin_intro}-\eqref{eqn:wong_term_intro}
is called the multivariable feedback particle filter.

The contributions of this paper are as follows:

\smallskip

\noindent {$\bullet$ \bf Consistency.}
The feedback particle filter~\eqref{eqn:particle_filter_nonlin_intro} is
consistent with the nonlinear filter,
\spm{I shortened this bullet}
given consistent initializations  $p(\varble,0)=p^*(\varble,0)$.
Consequently, if the initial conditions $\{X^i_0\}_{i=1}^N$  are drawn from the initial distribution $p^*(\varble,0)$ of
$X_0$, then, as $N\rightarrow\infty$, the empirical distribution of the particle system approximates the posterior distribution $p^*(\varble,t)$ for each $t$.
\smallskip

\noindent {$\bullet$ \bf Well-posedness.}
\spm{If $h$ has polynomial growth, and $|\log(p)|$ quadratic growth,
  then existence and uniqueness hold.   And, the growth rate of
  $\phi_j$ coincides with that of $h_j$}
\pgm{Perhaps easier and more intuitive conditions are in terms of Sobolev norms}
A weak formulation of~\eqref{eqn:EL_phi_intro} is introduced, and used
to prove an existence-uniqueness result for $\phi_j$ in a suitable
function space.  Certain apriori bounds are derived for the gain
function to show that the resulting  control input
in~\eqref{eqn:particle_filter_nonlin_intro} is admissible (That is,
the filter~\eqref{eqn:particle_filter_nonlin_intro} is well-posed in the It\^{o} sense).


\begin{figure*}
    \centering
 \includegraphics[scale=0.31]{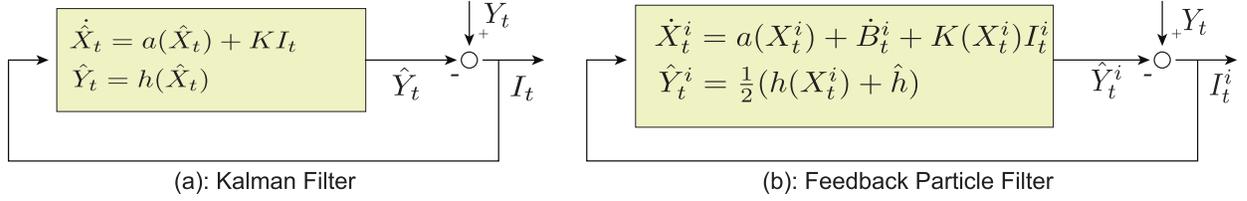}
   \vspace{-0.1in}
\caption{Innovation error-based feedback structure for (a) Kalman
  filter and (b) nonlinear feedback particle filter (see Remark~\eqref{rem:rem1}).
}\vspace{-.4cm}
    \label{fig:fig_FPF_KF}
\end{figure*}

\smallskip

\noindent {$\bullet$ \bf Numerical algorithms.}
Based on the weak formulation, a Galerkin finite-element algorithm is
proposed for approximation of the gain function $\v(x,t)$.  The
algorithm is completely adapted to data (That is, it does not require
an explicit approximation of $p(x,t)$ or computation of derivatives).
Certain closed-form expressions for gain function are derived in
certain special cases.
The conclusions are  illustrated with  numerical examples.

\noindent {$\bullet$ \bf Characterization of the feedback gain.}
The \textit{Smoluchowski equation} models  a $d$-dimensional gradient flow with ``noise'':
\begin{equation}
  \ud \Phi_t = - \nabla \smopot(\Phi_t) \ud t    +   \sqrt{2} \ud \xi_t,
\label{e:smo}
\end{equation}
where $\xi$ is a standard Wiener process.  It is regarded as the
original MCMC algorithm:  Under general conditions it is ergodic, with
(unnormalized) stationary distribution $e^{-\smopot}$.  
The BVP~\eqref{eqn:EL_phi_intro} can be expressed as an instance of
Poisson's equation for this diffusion,
\begin{equation}
\clD \phi_j = -(h_j-\hah_j)\,, \quad 1\le j\le m \,,
\label{e:fishphi}
\end{equation}
where   $\clD$ is the differential generator for the Smoluchowski equation, with potential  $\smopot=-\log p$.
Subject to growth conditions on $h$ and $p$,  this implies the mean-integral representation for the vector-valued function,
\begin{equation}
\phi_j(x) =  \int_0^\infty \Expect[ h_j(\Phi_t) -\hah_j \mid \Phi_0=x] \, \ud t.
\label{e:fishInt}
\end{equation}
This representation also suggests an alternate proof of well-posedness
and construction of numerical algorithms; cf.,~\cite{glymey96a}.
This will be the subject of future work.
\spm{I have a dozen papers on this -- not sure where to start}
\pgm{For well-posedness of {e:smo} one would require C^1 smoothness of
  \nabla \smopot; which in turn would require p is C^2 -- that is what
Rick assumed in his proofs}

\smallbreak




The outline of the remainder of this paper is as follows.  The
nonlinear filter is introduced and shown to be consistent
in~\Sec{sec:MFPF}.  The weak formulation of the BVP appears
in~\Sec{sec:wellposed} where well-posedness results are also derived.
Algorithms are discussed in~\Sec{sec:algo} and a numerical example in~\Sec{sec:numerics}.


\section{Multivariable Feedback Particle Filter}
\label{sec:MFPF}

Consider the continuous time filtering problem~\fp~introduced in \Sec{sec:intro}.

We denote as $p^*(x,t)$ the conditional distribution of $X_t$
given $\clZ_t = \sigma(Z_s:s\le t)$.  The evolution of   $p^*(x,t)$ is described by the Kushner-Stratonovich (K-S) equation:
\spm{Too kiss-ass:  Recall that the celebrated}
\begin{equation}
\ud p^\ast = \clL^\dagger p^\ast \ud t + ( h-\hat{h} )^T
(\ud Z_t - \hat{h} \ud t)p^\ast,
\label{eqn:Kushner_eqn}
\end{equation}
where $ \hat{h} = \int h(x) p^*(x,t) \ud x$ and $\clL^\dagger p^\ast =
- \nabla \cdot (pa) + \frac{1}{2}\Delta p$, where $\Delta$ denotes the
Laplacian in $\Re^d$.

\subsection{Belief state dynamics \&\ control architecture}

The model for the particle filter is
given by,
\begin{equation}
\ud X^i_t = a(X^i_t) \ud t + \ud B^i_t +  \underbrace{u(X^i_t, t) \ud t + \v(X^i_t,t) \ud Z_t}_{\ud U_t^i} , \label{eqn:particle_model}
\end{equation}
where $X^i_t \in \Re^d$ is the state for the $i^{\text{th}}$
particle at time $t$, and $\{B^i_t\}$ are mutually independent
standard Wiener processes.  We assume the initial conditions
$\{X^i_0\}_{i=1}^N$  are i.i.d., independent of $\{B^i_t\}$,
and drawn from the initial
distribution $p^*(x,0)$ of $X_0$.  Both  $\{B^i_t\}$ and $\{X^i_0\}$
are also assumed to be independent of $X_t,Z_t$. Note that the gain function $\v(x,t)$ is a $d\times m$ matrix and $u(x,t)\in\Re^d$.

We impose admissibility requirements on
the control input $U^i_t$ in~\eqref{eqn:particle_model}:

\begin{definition}[Admissible Input]
The control input $U^i_t$ is {\em admissible} if the random variables $u(x,t) $
and $\v(x,t)$ are  $\clZ_t = \sigma(Z_s:s\le t)$ measurable for each
$t$.   And for each $t$, $ \Expect[|u|] := \Expect[ \sum_{l} |
u_l(X^i_t,t) | ]<\infty$ and $\Expect[|\v|^2] := \Expect [ \sum_{lj} |
\v_{lj}(X^i_t,t) |^2 ]<\infty$.
\qed
\end{definition}

Recall that there are two types of conditional distributions of
interest in our analysis:
\begin{enumerate}
\item $p(x,t)$:  Defines the conditional dist.\ of $X^i_t$
    given $\clZ_t$.
\item $p^*(x,t)$: Defines the conditional dist.\ of $X_t$
    given $\clZ_t$.
\end{enumerate}
The functions $\{ u(x,t),\v(x,t)\}$ are said to be
\textit{optimal} if $p\equiv p^*$.  That is,
given $p^*(\cdot,0)= p(\cdot,0)$, our goal is to choose
$\{u,\v\}$ in the feedback particle filter so that the
evolution equations of these conditional distributions coincide
(see \eqref{eqn:Kushner_eqn} and \eqref{eqn:mod_FPK}).

The evolution equation for the belief state is described in the next
result.  The proof is identical to the proof
in the scalar case (see Proposition~2 in~\cite{YangMehtaMeyn11CDC}).
It is omitted here.
\begin{proposition}
\label{thm:FPK}
Consider the process $X^i_t$ that evolves according to the particle
filter model~\eqref{eqn:particle_model}.  The conditional distribution
of $X^i_t$ given the filtration $\clZ_t$, $p(x,t)$, satisfies
the forward equation
\begin{equation}
\begin{aligned}
\ud p = \clL^\dagger p \ud t  &- \nabla \cdot (p\v) \ud Z_t
\\
&- \nabla \cdot (pu) \ud t +
\frac{1}{2}\sum_{l,k=1}^d \frac{\partial^2}{\partial x_l \partial
  x_{k}} \left( p [\v
  \v^T]_{lk} \right) \ud t.
  \end{aligned}
\label{eqn:mod_FPK}
\end{equation}
\qed
\end{proposition}

\subsection{General Form of the Feedback Particle Filter}

The general form of the feedback particle filter is obtained by
choosing $\{u,\v\}$ as the solution to a certain E-L BVP based on
$p$.  The function $\v$ is a solution to
\begin{equation}
\nabla \cdot (p \v) = - (h-\hat{h})^T
 p,
\label{e:bvp_divergence_multi}
\end{equation}
and the function $u$ is obtained as
\spm{W should be $\w$, right??}
\begin{equation}
u(x,t) = -\frac{1}{2} \v(x,t) \left( h(x) + \hat{h} \right) + \w(x,t).
\label{eqn:u_intermsof_v*}
\end{equation}
The reader is referred to our earlier paper~\cite{YangMehtaMeyn11CDC} for additional justification regarding these choices.
\spm{choices?
\\
And, this  is not useful:
-- where the E-L
BVP~\eqref{e:bvp_divergence_multi} is obtained starting from a
variational formulation -- .
\\
and, the remark isn't very helpful -- we can reduce later
}

\pgm{The remark is useful for other papers, where we work with ODE
  formulation.  Note that a lot of (most?) users of filtering are not really
  comfortable with SDEs.  The related Stratonovich formulation may also be more
  suitable for implementation.}

\begin{remark}\label{rem:rem1}
Substituting~\eqref{e:bvp_divergence_multi}-\eqref{eqn:u_intermsof_v*}
into~\eqref{eqn:particle_model} gives the feedback particle filter
model~\eqref{eqn:particle_filter_nonlin_intro}-\eqref{e:in_intro} in
\Sec{sec:intro}.

In the Stratonovich form, the filter admits a simpler representation,
\begin{equation}
\ud X^i_t = a(X^i_t) \ud t + \ud B^i_t +  \v(X^i,t) \circ \left( \ud Z_t - \frac{1}{2}    (h(X^i_t) + \hat{h}) \ud t\right). \nonumber
\end{equation}
Given that the Stratonovich form provides a mathematical
interpretation of the (formal) ODE model \cite[Section 3.3 of the SDE
text by {\O}ksendal]{Oksendal_book}, we also obtain the (formal) ODE
model of the filter. Denoting $Y_t \doteq \frac{\ud Z_t}{\ud t} $ and white noise
process $\dot{B}^i_t \doteq \frac{\ud B_t^i}{\ud t} $, the ODE model of the
filter is given by,
\begin{equation}
\frac{\ud X^i_t}{\ud t} = a(X^i_t) + \dot{B}^i_t +  \v(X^i,t) \cdot \left( Y_t - \frac{1}{2}    (h(X^i_t) + \hat{h})\right). \nonumber
\end{equation}
The feedback particle filter thus provides   a generalization of the
Kalman filter to nonlinear systems, where the innovation error-based
feedback structure of the control is preserved (see
Fig.~\ref{fig:fig_FPF_KF}).  For the linear case, it is shown in
\Sec{sec:LG} that the gain function
is the Kalman gain.  For the nonlinear case, the Kalman gain is replaced by a nonlinear function of the state.
\qed
\end{remark}
\spm{Note that I changed QED to qed for a reason -- the QED symbols was all over the place!}

If one further {\em assumes} that the control input $U^i_t$ is
admissible, a short calculation shows that the feedback particle
filter is consistent with the choice of $\{u,\v\}$ given
by~\eqref{e:bvp_divergence_multi}-\eqref{eqn:u_intermsof_v*}.  This
calculation appears in Appendix~\ref{apdx:pf_Kushner}.

\spm{I do not agree with this:
The remainder of
the paper thus is about showing that $U^i_t$ is admissible.  This
requires additional assumptions.
\\
I moved part of the statement to the next subsection
}




\subsection{Consistency with the Nonlinear Filter}

To establish admissibility of the input $U^i_t$ requires additional assumptions on the density $p$ and function $h$:
\spm{Too bad we have to assume smoothness -- The paper  {konmey09a} has references on smoothness
}
\pgm{Note to Sean: {knomey09a} may also require $C^1$ smoothness of the drift term
  \nabla \smopot(\Phi)?  This is the same as $\smopot$ below being $C^2.$
  \\ 
  PGM:  We justify this following Kunita}
\begin{romannum}
\item {\textbf{\textit{Assumption~A1}}}
The probability density $p(x,t)$ is of the form $p(x,t) = e^{-\smopot(x,t)}$,
where $\smopot(x,t)$ is a twice continuously
differentiable function 
with
\begin{equation} \label{eqn:Poincare_assump}
|\nabla \smopot|^2 - 2\Delta \smopot \rightarrow \infty \quad \text{as $|x| \rightarrow \infty$.}
\end{equation}

\spm{$D^2$ and $\nabla$ looks a bit ugly.  What's wrong with $\nabla^2$?
This only comes up 3 times, so it isn't a big deal}
\pgm{$\nabla^2$ invites confusion with Laplacian $\Delta$.  Old story.}

\item {\textbf{\textit{Assumption~A2}}}
The function $h$ satisfies,
\[
\int |h(x)|^2 p(x,t) \ud x < \infty,
\]
where $|h(x)|^2:= \sum_{j} |h_j(x)|^2$.
For admissibility of $u$, our arguments require additional assumptions:
\pgm{Note to Rick:  Regularity theory on finite domain sugest
  A3 and A4 are not really needed.  May be include a remark following
  Theorem 1?}

\item {\textbf{\textit{Assumption~A3}}}
The second derivatives of $\smopot(x,t)$ with respect to $x$ are
uniformly bounded at each $t$, i.e., $|\frac{\partial^2
  \smopot}{\partial x_j \partial x_k}(x,t)| \le c_2(t)$ 
for all $x\in\Re^d$, $t>0$.
\item {\textbf{\textit{Assumption~A4}}}
The first (weak) derivatives of $h$ satisfy
\[
\int |\nabla h(x)|^2 p(x,t) \ud x < \infty,
\]
where $|\nabla h(x)|^2:= \sum_{jk} |\frac{\partial h_j}{\partial x_k} (x)|^2$.
\qed
\pgm{Notation 1, 2 norm will have to be centralized; this is not efficient.}
\end{romannum}

Under these assumptions, it is shown in Theorem~\ref{thm:thm1} that the gradient-form \eqref{eqn:EL_phi_intro}
of the E-L BVP~\eqref{e:bvp_divergence_multi} is uniquely obtained to
give $\phi$ and thence $\v$.
\spm{This seemed odd,
"a
particular `gradient-form'
solution"
given this in the next sentences,
"The BVP for the gradient form solution has already appeared
in~\eqref{eqn:EL_phi_intro} in \Sec{sec:intro}. "
" In this case, $\v$
obtained as~\eqref{eqn:gradient_gain_fn_intro} is a solution of
the E-L BVP~\eqref{e:bvp_divergence_multi}. "
\\
I've attempted to tighten it up.
}

The admissibility of the resulting control input is established in Corollary~\ref{cor:cor1}.
Theorem~\ref{thm:thm1} and Corollary~\ref{cor:cor1} are stated and proved in \Sec{sec:wellposed}.

The following theorem then shows that the two evolution equations
\eqref{eqn:Kushner_eqn} and \eqref{eqn:mod_FPK} are identical.
The proof appears in Appendix~\ref{apdx:pf_Kushner}.

\begin{theorem}\label{thm:kushner}
Consider the two evolution equations for $p$ and $p^*$, defined
according to the solution of the forward
equation~\eqref{eqn:mod_FPK} and the K-S
equation~\eqref{eqn:Kushner_eqn}, respectively.  Suppose that
the gain function $\v(x,t)$ is obtained
according to~\eqref{eqn:EL_phi_intro}-\eqref{eqn:gradient_gain_fn_intro}.
Then, provided $p(\varble,0)=p^*(\varble ,0)$, we have for all $t\ge 0$,
\[
p(\varble,t) = p^*(\varble,t).
\]
\qed
\end{theorem}

\section{Existence, Uniqueness and Admissibility}
\label{sec:wellposed}

The aim of this section is to introduce a particular {\em
  gradient-form solution} of the BVP~\eqref{e:bvp_divergence_multi}.
The gradient-form solution is obtained in terms of $m$ real-valued
functions $\{\phi_1(\varble,t),\phi_2(\varble,t),\hdots,\phi_m(\varble,t)\}$.  For
$j=1,2,\hdots,m$, the function $\phi_j$ is a solution to,
\begin{equation}
\begin{aligned}
\nabla \cdot (p(x,t) \nabla \phi_j(x,t) ) & = - (h_j(x)-\hat{h}_j)p(x,t),
\\
\int \phi_j(x,t) p(x,t) \ud x & = 0.
\label{eqn:EL_phi}
\end{aligned}
\end{equation}
The normalization $\int \phi_j(x,t) p(x,t) \ud x   = 0  $ is for
convenience:  If $\phi_j^o$ is an solution to the differential equation \eqref{eqn:EL_phi},
we obtain the desired normalization  on subtracting its mean.
\spm{needed to explain normalization}

In terms of these solutions, the gain function is given by,
\begin{equation}
\v_{lj} (x,t)=  \frac{\partial \phi_j} {\partial x_l} \,(x,t)\,, \quad x\in\Re^d\, .
\label{eqn:gradient_gain_fn}
\end{equation}
It is straightforward to verify that $\v$ thus defined is a particular
solution of the BVP~\eqref{e:bvp_divergence_multi}.
\ty{Not sure if it's correct: $\nabla \phi$ is $1\times d$ matrix?
\\
 phi is an m-dim vector if there are m observations --spm}

\subsection{Poisson's Equation Interpretation}

The differential equation \eqref{eqn:EL_phi} is solved for each $t$ to give the $m$ functions $\{\phi_j (\varble,t) : 1\le j\le m\}$.
On dividing each side of this equation by $p$,    elementary calculus leads to the equivalent equation
\eqref{e:fishphi}, with generator $\clD$ defined for $C^2$ functions $f$ via,
\begin{equation}
\clD f = - \nabla \smopot  \cdot \nabla f   +   \Delta f
\nonumber
\end{equation}
and with $\smopot(\varble)=-\log p(\varble,t)$.
This is the differential generator for the Smoluchowski equation
\eqref{e:smo}.  It is shown in \cite{huimeysch04a} that this diffusion
is exponentially ergodic under mild conditions on $\smopot$.
Consequently,  $ \Expect[ h_j(\Phi_t) -\hah_j \mid \Phi_0=x] $
converges to zero exponentially fast, subject to growth conditions on
$h_j$,  and from this we can conclude that \eqref{e:fishInt} is well
defined, and provides a solution to Poisson's equation
\eqref{e:fishphi} \cite{glymey96a}.

Poisson's equation can be regarded as the \textit{value function} that
arises in average-cost optimal control, and this is the object of
interest in the approximation techniques used in TD-learning for
average-cost optimal control \cite{CTCN}.  The integral representation
\eqref{e:fishInt} suggests approximation techniques based on
approximate models for the diffusion $\bfPhi$.
\spm{What more can I put in from my notes?}

\medbreak

The remainder of this section is devoted to showing existence and
uniqueness of the solution of \eqref{eqn:EL_phi}, and admissibility of
the resulting control input, obtained using gain function defined
by~\eqref{eqn:gradient_gain_fn}.

\subsection{Weak Formulation}

Further analysis of this problem requires introduction of Hilbert
spaces: $L^2(\Re^d;p)$ is used to denote the Hilbert space of functions
on $\Re^d$ that are square-integrable with respect to density
$p(\cdot,t)$ (for a fixed time $t$);
$H^k(\Re^d;p)$ is used to denote the Hilbert space of
functions whose first $k$-derivatives (defined in the weak sense)
are in $L^2(\Re^d;p)$.
Denote
\[
H_0^1(\Re^d;p) := \left\{ \phi\in H^1(\Re^d;p) \,\Big|\, \int
\phi(x) p(x,t) \ud x =0 \right\}.
\]

A function $\phi_j \in H_0^1(\Re^d;p)$ is said to be a weak solution of
the BVP~\eqref{eqn:EL_phi} if
\begin{equation}
\int \nabla \phi_j(x,t) \cdot \nabla \psi(x) p(x,t) \ud x = \int  (h_j(x)-\hat{h}_j) \psi(x) p(x,t) \ud x,\label{eqn:EL_phi_weak}
\end{equation}
for all $\psi \in H^1(\Re^d;p)$.

Denoting $\Expect [\cdot]:= \int \cdot p(x,t) \ud x$, the weak form of the
BVP~\eqref{eqn:EL_phi}  can also be expressed as
\begin{equation}
\Expect [\nabla \phi_j \cdot \nabla \psi ] = \Expect[  (h_j
-\hat{h}_j) \psi],\quad\forall\,\,\psi\in H^1(\Re^d;p).
\label{eqn:EL_phi_expect}
\end{equation}
This representation is useful for the numerical algorithm described in
\Sec{sec:algo}.

\subsection{Main Results}


The existence-uniqueness result for the BVP~\eqref{eqn:EL_phi} is
described next --- Its proof is given  in Appendix~\ref{apdx:uniqueness}.

\begin{theorem}\label{thm:thm1}
Under Assumptions A1-A2, the BVP~\eqref{eqn:EL_phi} possesses a unique weak solution $\phi_j \in H_0^1(\Re^d;p)$, satisfying
\begin{equation}
\int |\nabla\phi_j(x)|^2 p(x,t) \ud x  \le \frac{1}{\lambda} \int |h_j(x) - \hat{h}_j|^2 p(x,t)
\ud x\label{eqn:bound1} .
\end{equation}
If in addition Assumptions A3-A4 hold, then $\phi_j \in H^2(\Re^d;p)$ with
\begin{equation}
\int  \left| (D^2 \phi_j) (\nabla \phi_j) \right|  p(x,t) \ud x  \le
C(\lambda;p) \int |\nabla h_j|^2 p(x,t) \ud x,
\label{eqn:bound2}
\end{equation}
where $\lambda$ is (spectral gap) constant (see Appendix~\ref{apdx:uniqueness}) and $C(\lambda;p)=\frac{1}{\lambda^{3/2}} \left( \frac{\lVert D^2(\log p)
    \rVert_{L^\infty}}{\lambda} + 1 \right)^{\! 1/2}$.
\qed
\end{theorem}


The apriori bounds~\eqref{eqn:bound1}-\eqref{eqn:bound2} are used to
show that the control input for the feedback particle filter is
admissible. The proof is omitted on account of space.
\begin{corollary}\label{cor:cor1}
Suppose $\phi_j$ is the weak solution of BVP~\eqref{eqn:EL_phi} as
described in Theorem~\ref{thm:thm1}.  The gain function $\v$ is
obtained using~\eqref{eqn:gradient_gain_fn} and $u$ is given
by~\eqref{eqn:u_intermsof_v*}.  Then
\begin{align*}
\Expect[|\v|^2] & \le  \frac{1}{\lambda} \sum_{j=1}^m \int |h_j(x)|^2
p(x,t) \ud x,\\
\Expect[ |u| ] & \le \left( \frac{1}{\lambda} +
  C(\lambda;p) \right)  \sum_{j=1}^m \int \left( |h_j(x)|^2 + |\nabla h_j|^2 \right)
p(x,t) \ud x,
\end{align*}
where $C(\lambda;p)$ is given in Theorem~\ref{thm:thm1}.  That is, the resulting control input in~\eqref{eqn:gradient_gain_fn} is admissible.
\spm{$|R|^2$ changed to $m^2$}
\qed
\end{corollary}

\spm{I don't think \Sec{sec:LG} is very useful}
\subsection{Linear Gaussian case}
\label{sec:LG}
Consider the linear system,
\begin{subequations}
\begin{align}
\ud X_t &= A\; X_t \ud t + \ud B_t \label{eqn:multi1}\\
\ud Z_t &= H\; X_t \ud t + \ud W_t \label{eqn:multi2}
\end{align}
\end{subequations}
where $A$ is an $d\times d$ matrix, and $H$ is an $m \times d$
matrix. The initial distribution $p^\ast (x,0)$ is Gaussian with mean vector
$\mu_0$ and covariance matrix $\Sigma_0$.

The following proposition shows that the Kalman gain is a
gradient-form solution of the multivariable BVP~\eqref{e:bvp_divergence_multi}:

\begin{proposition}
\label{prop_multi} Consider the d-dimensional linear
system~\eqref{eqn:multi1}-\eqref{eqn:multi2}.  Suppose $p(x,t)$ is
assumed to be Gaussian: $p(x,t) = \frac{1}{(2\pi)^{\frac{d}{2}}
  |\Sigma_t|^\frac{1}{2}} \exp \left(-\frac{1}{2}(x-\mu_t)^T
  \Sigma_t^{-1} (x-\mu_t)\right)$, where $x = (x_1,x_2,...,x_d)^T$,
$\mu_t=(\mu_{1t},\mu_{2t},\hdots,\mu_{dt})^T$ is the mean,
$\Sigma_t$ is the covariance matrix, and $|\Sigma_t|>0$ denotes the
determinant.  A solution of the BVP~\eqref{eqn:EL_phi} is given by,
\begin{equation}
\phi_j(x,t) = \sum_{k=1}^d [\Sigma_t H^T]_{kj} (x_k - \mu_{kt}).
\label{eqn:linsol_v}
\end{equation}
Using~\eqref{eqn:gradient_gain_fn}, $\v(x,t)=\Sigma_t H^T
$ (the Kalman gain) is the gradient form solution
of~\eqref{e:bvp_divergence_multi}.
\qed
\end{proposition}

The formula~\eqref{eqn:linsol_v} is verified by direct
substitution in the BVP~\eqref{eqn:EL_phi} where the distribution $p$ is
multivariable Gaussian.

The gain function yields the following form for the particle filter in this linear Gaussian model:
\begin{equation}
\begin{aligned}
\ud X^i_t= A \; X^i_t \ud t +
 \ud B^i_t + \Sigma_t H^T
 \left( \ud Z_t - H \frac{X^i_t + \mu_t}{2} \ud t \right).
\end{aligned}
\label{eqn:particle_filter_lin}
\end{equation}

Now we show that $p=p^*$ in this case. That is, the conditional
distributions of $\bfmX$ and $\bfmX^i$ coincide, and are
defined by the well-known dynamic equations that characterize
the mean and the variance of the continuous-time Kalman filter.

\begin{theorem}
\label{thm_lin} Consider the linear Gaussian filtering problem
defined by the state-observation equations~\eqref{eqn:multi1}-\eqref{eqn:multi2}.  In this case the
posterior distributions of $X_t$ and $X^i_t$  are Gaussian,
whose conditional mean and covariance are given by the
respective SDE and the ODE,
\begin{align*}
 \ud \mu_t  &= A \mu_t \ud t + \Sigma_t H^T
  \Bigl(\ud Z_t- H\mu_t \ud t \Bigr)
\\[.1cm]
\ddt \Sigma_t
&= A \Sigma_t + \Sigma_t A^T + I - \Sigma_t H^T
H \Sigma_t
\end{align*}
\qed
\end{theorem}

The result is verified by substituting $p(x,t)= (2\pi)^{-\frac{d}{2}}|\Sigma_t|^{-\frac{1}{2}}\exp\left[-\frac{1}{2}(x-\mu_t)^T\Sigma_t^{-1}(x-\mu_t)\right]$ in the forward
equation~\eqref{eqn:mod_FPK}.
The details are omitted on account of space, and because the result is a special case of Theorem~\ref{thm:kushner}.


In practice $\{\mu_t, \Sigma_t\}$ are approximated as sample means and sample covariances from the ensemble
$\{X^i_t\}_{i=1}^N$:
\spm{Someone revise!  Shorten, and and also remember that Sigma is a matrix}
\begin{align*}
\mu_t & \approx \mu_t^{(N)} := \frac{1}{N} \sum_{i=1}^N X^i_t,\\
\Sigma_t & \approx \Sigma_t^{(N)} := \frac{1}{N-1} \sum_{i=1}^N (X^i_t - \mu_t^{(N)})^2.
\end{align*}
The resulting equation
\eqref{eqn:particle_filter_lin} for the $i^{\text{th}}$
particle is given by
\begin{equation}
\begin{aligned}
\ud X^i_t= A \; X^i_t \ud t +
 \ud B^i_t + \Sigma_t^{(N)}H^T
 \left( \ud Z_t - H \frac{X^i_t + \mu_t^{(N)}}{2} \ud t \right).
\end{aligned}
\nonumber
\end{equation}
As $N\rightarrow \infty$, the empirical distribution of the particle
system approximates the posterior distribution $p^{*}(x,t)$ (by Theorem~\ref{thm_lin}).



\section{Finite-Element Algorithm}
\label{sec:algo}

In this section, a Galerkin finite-element algorithm is described to
construct  an approximate solution of~\eqref{eqn:EL_phi_weak}.  Since
there are $m$ uncoupled BVPs, without loss of generality, we assume
scalar-valued observation in this section, with $m=1$, so that
$\v=\nabla\phi$.  The time $t$ is fixed.  The explicit
dependence on time is suppressed for notational ease (That is,
$p(x,t)$ is denoted as $p(x)$, $\phi(x,t)$ as $\phi(x)$ etc.).

\spm{repetitive:
The method proposed here can be used to obtain approximate solution
of the BVP~\eqref{eqn:EL_phi_weak} for $j=1,2,\hdots,m$ each.  The gain
function for the general $m$-dimensional observation case is then
obtained using~\eqref{eqn:gradient_gain_fn}.
}

\def\prechi{\zeta}

\subsection{Galerkin Approximation}

Using~\eqref{eqn:EL_phi_expect}, the gain function $\v=\nabla\phi$ is a weak solution if
\begin{equation}
\Expect[ \v \cdot \nabla \psi ] =
\Expect [(h-\hat{h}) \psi ],\quad\forall\,\,\psi\in H^1(\Re^d;p).\label{eqn:BVP_phi_weak_expect}
\end{equation}
The gain function is approximated as,
\begin{equation}
\v = \sum_{l=1}^L \kappa_l \chi_l(x),\nonumber
\end{equation}
where $\{\chi_l(x) \}_{l=1}^L$ are {\em basis functions}.  For each $l
=1,\hdots,L$, $\chi_l(x)$ is a gradient function; That is, $\chi_l(x)
= \nabla \prechi_l(x)$ for some function $\prechi_l(x)\in H_0^1(\Re^d;p)$.
\spm{someone had used $\phi$ for $\prechi$.  No! no!!}
\pgm{All symbols look the same to me.  ----ing Greeks.}

The {\em test functions} are denoted as $\{\psi_k(x)\}_{k=1}^L$ and $S:=\text{span}\{\psi_1(x),\psi_2(x),\hdots,\psi_L(x)\}\subset H^1(\Re^d;p)$.
\spm{I do not understand:  What is $H^1$?
Shouldn't $S$ be a function space?}
\pgm{corrected}

The finite-dimensional approximation of the
BVP~\eqref{eqn:BVP_phi_weak_expect} is to choose constants
$\{\kappa_l\}_{l=1}^L$ such that
\begin{equation}
\sum_{l=1}^L \kappa_l \Expect[ \chi_l \cdot \nabla \psi ] =
\Expect [  (h-\hat{h}) \psi ],\quad\forall\,\,\psi\in
S.
\label{eqn:BVP_phi_weak_expect_fd}
\end{equation}

Denoting $[A]_{kl} = \Expect[ \chi_l \cdot
\nabla \psi_k ]$, $b_k = \Expect [  (h-\hat{h}) \psi_k ]$,
$\kappa=(\kappa_1,\kappa_2,\hdots,\kappa_L)^T$, the
finite-dimensional approximation~\eqref{eqn:BVP_phi_weak_expect_fd}
is expressed as a linear matrix equation:
\[
A\kappa = b.
\]
The matrix $A$ and vector $b$ are easily approximated by using only the particles:
\begin{align}
[A]_{kl} & = \Expect[ \chi_l \cdot
\nabla \psi_k ] \approx \frac{1}{N} \sum_{i=1}^N \chi_l(X^i_t) \cdot
\nabla \psi_k (X^i_t),\label{formula:A_ml_sample}\\
b_k & = \Expect [  (h-\hat{h}) \psi_k ]  \approx \frac{1}{N}
\sum_{i=1}^N (h(X^i_t) - \hat{h}) \psi_k(X^i_t), \label{formula:b_m_sample}
\end{align}
where recall $\hat{h} \approx \frac{1}{N}
\sum_{i=1}^N h(X^i_t)$.



\spm{the examples confuse me.  I stopped editing this section, and jumped to the appendix.}

\subsection{Example 1: Constant Gain Approximation}

Suppose $\chi_l=e_l$, the canonical coordinate vector with value $1$
for the $l^{\text{th}}$ coordinate and zero otherwise.  The test
functions are the coordinate functions $\psi_k(x) = x_k$ for
$k=1,2,\hdots,d$.
Denoting $\psi(x) = (\psi_1,\psi_2,\hdots,\psi_d)^T = x$,
\begin{align}
\kappa = \Expect[\v] = \Expect [  (h-\hat{h}) \psi ]  &= \int (h(x)
-\hat{h}) \psi (x) p(x) \ud x \nonumber\\
&\approx \frac{1}{N} \sum_{i=1}^N (h(X^i_t) - \hat{h}) X^i_t.\label{eqn:const_approx_gain}
\end{align}
This formula yields the {\em constant-gain approximation} of the gain function.

\subsection{Example 2: Single-state Case}

Consider a scalar example, where the density is a sum of Gaussian,
\begin{equation}
p(x) \approx \sum_{j=1}^3 \lambda^j q^j(x),
\nonumber
\end{equation}
where $q^j(x) = q(x;\mu^j,\Sigma^j)= \frac{1}{\sqrt{2\pi \Sigma^j}}
\exp(-\frac{(x-\mu^j)^2}{2\Sigma^j})$,  $\lambda^j >0$, $\sum
\lambda^j =1$.  The parameter values for $\lambda^j,\mu^j, \Sigma^j$
are tabulated in Table~I.

\begin{figure}
\centering
\includegraphics[width=\hsize]{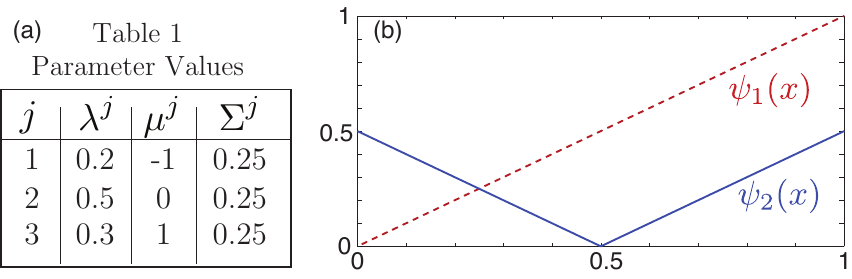}
\caption{(a) Parameter values and (b)
  $(\psi_1,\psi_2)$ in the Example.}
\label{fig:base_fnc}
\end{figure}

In the scalar case, a direct numerical solution (DNS) of the gain function is obtained by numerically approximating the integral
\begin{equation*}
\v(x) = -\frac{1}{p(x)} \int_{-\infty}^x (h(y)-\hat{h}) p(y) \ud y.
\end{equation*}
The DNS solution is used to provide comparisons with the approximate Galerkin
solutions.

The Galerkin approximation of the gain function is constructed on an
interval domain $D\subset\Re$.  The domain is a union of finitely
many non-intersecting intervals $D_l = [a_{l-1},a_l)$, where
$a_0<a_1<\hdots<a_L$.

Define for $l=1,2,\hdots,L$ and $k=1,2,\hdots,L$:
\begin{align*}
\text{Basis functions:} & \quad  \quad \chi_l(x) = 1_{D_l}(x),\\
\text{Test functions:} & \quad  \quad \psi_k(x)  = |x-a_{k}|.
\end{align*}
Figure~\ref{fig:base_fnc} depicts the test functions $\{\psi_1(x),\psi_2(x)\}$ for
$D=[0,1]$ and $a_0=0$, $a_1 = \frac{1}{2}$ and $a_2=0$.
The basis functions are the indicator functions on $[0,\frac{1}{2})$ and
$[\frac{1}{2},1)$.


Figure~\ref{fig:fig_h_x2} depicts a comparison of the DNS solution
and the Galerkin solution for $h(x)=x^2$, $D=[-2,2]$ and $L=1,5,15$.
For a given $L$, the basis and test functions are constructed for a
uniform partition of the domain (That is, $a_{l}= -2 + \frac{l}{L}4$).
The Galerkin solution is obtained using $N=1000$ particles that
are sampled from the distribution $p$.  The particles are used to
compute matrix $A$ and vector $b$, using
formulae~\eqref{formula:A_ml_sample} and \eqref{formula:b_m_sample},
respectively.  Since the analytical form of $p$ is known, these
matrices can also be assembled by using the integrals:
\begin{align}
[A]_{kl} & = \int  \chi_l (x) \cdot
\nabla \psi_k(x) p(x) \ud x,\label{formula:A_ml_integral}\\
b_k & = \int
(h(x)-\hat{h}) \psi_k(x) p(x) \ud x. \label{formula:b_m_integral}
\end{align}
The figure also depicts the Galerkin solution based on the integral
evaluation of the matrix $A$ and vector $b$.


For $L=15$, the matrix $A$ was found to be singular for the
particle-based implementation.  This is because there are no particles
in $D_{15}$.  In this case, the Galerkin solution is obtained using
only the integral
formulae~\eqref{formula:A_ml_integral}-\eqref{formula:b_m_integral}.
These formulae are exact while the particle-based
formulae~\eqref{formula:A_ml_sample} and \eqref{formula:b_m_sample}
are approximations.  In the other two cases ($L=1$ and $L=5$),
the particle-based solution provides a good approximation.


\begin{figure*}
\centering \hspace{-0.2cm}
\includegraphics[scale=1]{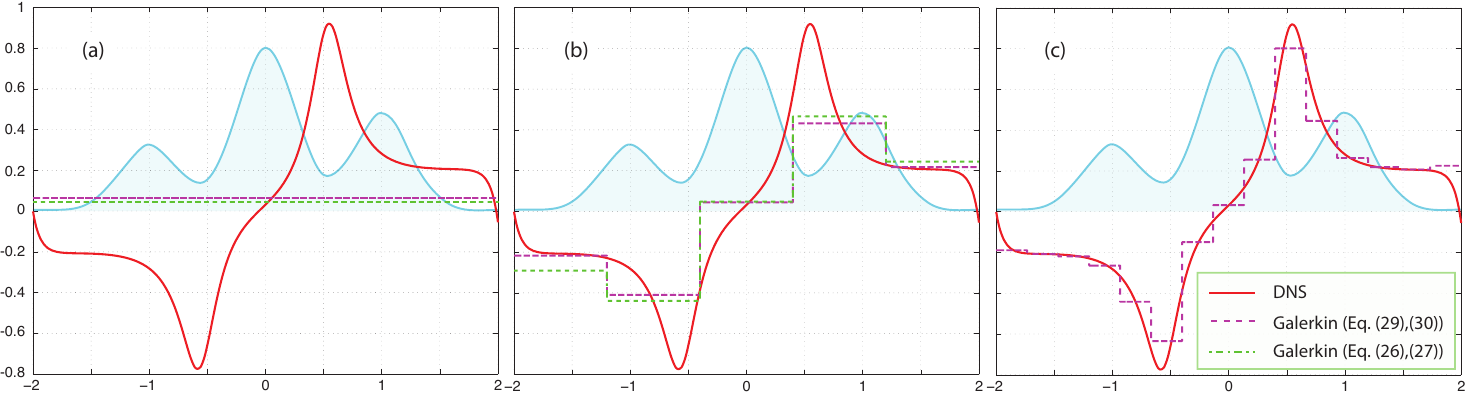}
\vspace{-0.1in}
 \caption{Comparison of the DNS and the Galerkin approximations of the gain
   function for $h(x)=x^2$ and: (a) $L=1$, (b)
   $L=5$ and (c) $L=15$.  The density is depicted as the shaded curve
 in the background.} \vspace{0.25cm} \label{fig:fig_h_x2}
\vspace{-0.2in}
\end{figure*}



\section{Numerics}
\label{sec:numerics}


Consider a target tracking problem with two bearing-only
sensors~\cite{Bar-Shalom:2002:EAT:863669}.  A single target moves in a
two-dimensional (2d) plane  according to the standard white-noise acceleration model:
\begin{equation}
\ud X_t = A X_t \ud t + \Gamma \ud B_t,\nonumber
\end{equation}
where $X :=(X_1,V_1,X_2,V_2)^T\in \Re^4$, $(X_1,X_2)$ denotes the
position and $(V_1,V_2)$ denotes the velocity.  The matrices, 
\begin{equation}
A = \begin{bmatrix}
	0 & 1 & 0 & 0\\
	0 & 0 & 0 & 0\\
	0 & 0 & 0 & 1\\
	0 & 0 & 0 & 0
	\end{bmatrix},\qquad \Gamma = \sigma_B \begin{bmatrix}
										0 & 0\\
										1 & 0\\
										0 & 0\\
										0 & 1
										\end{bmatrix},\nonumber
\end{equation}   
and $B_t$ is a standard 2d Wiener process.

The observation model is given by,     
\begin{equation}
\ud Z_t = h(X_t) \ud t + \sigma_W \ud W_t,\nonumber
\end{equation}
where $W_t$ is a standard 2d Wiener process, $h = (h_1,h_2)^T$ and
\begin{equation}
h_j(x_1,v_1,x_2,v_2) = \arctan \left(\frac{x_2 - x_2^{(\text{sen}\; j)}}{x_1 - x_1^{(\text{sen}\; j)}}\right),\quad j = 1,2,\nonumber
\end{equation}
where $(x_1^{(\text{sen}\;
  j)},x_2^{(\text{sen}\; j)})$ denote the position of sensor $j$. 

Figure~\ref{fig:target_track_mfpf} depicts a sample path obtained for
a typical numerical experiment.  The sensor and target locations are
depicted together with an estimate (conditional mean) that is
approximated using a feedback particle filter.  The background depicts
the ensemble of observations that were made over the simulation run.  Each point in
the ensemble is obtained by using the process of triangulation based
on two (noisy) angle measurements.  The simulation parameters are: The
initial position of the target is depicted, the initial velocity was
chosen as $(0.2,-5)$ and $\sigma_B = 0.1$; The two sensor positions
are depicted and $\sigma_W = 0.017$; The particle filter comprised of
$N=200$ particles whose initial position was chosen from a Gaussian
distribution whose mean is depicted.  The gain function was obtained
using the constant gain approximation
in~\eqref{eqn:const_approx_gain}.  The simulation results show that the
filter can adequately track the target.

\begin{figure}
\centering
\includegraphics[scale=1]{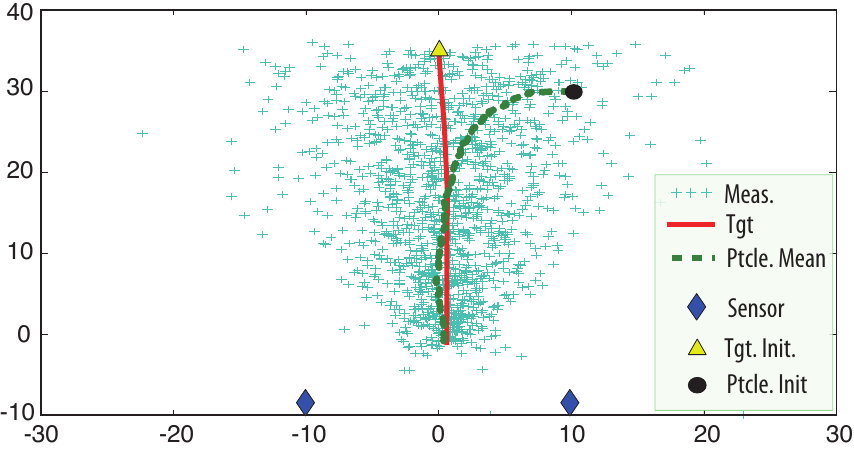}
\vspace{-0.0in}
\caption{Simulation results: Comparison of the true target trajectory
  with the estimate obtained using FPF.}
\vspace{-.25cm}
\label{fig:target_track_mfpf}
\end{figure}



\appendix
\section{Appendix}

\subsection{Proof of Theorem~\ref{thm:kushner}}
\label{apdx:pf_Kushner}
It is only necessary to show that with the choice of $\{u,\v\}$ given
by~\eqref{e:bvp_divergence_multi}-\eqref{eqn:u_intermsof_v*},  we have
$\ud p(x,t) = \ud p^*(x,t)$, for all $x$ and $t$,  in the sense that
they are defined by identical stochastic differential equations.   Recall $\ud p^*$ is
defined according to the K-S equation~\eqref{eqn:Kushner_eqn},
and $\ud p$ according to the forward equation~\eqref{eqn:mod_FPK}.

Recall that the gain function $\v$ is a solution of the following BVP:
\begin{equation}
\nabla \cdot (p\v) = -p(h-\hat{h})^T \, .
\label{eqn:bvp_multi_Z_matrix}
\end{equation}
On multiplying both sides of~\eqref{eqn:u_intermsof_v*} by $-p$, we obtain
\begin{equation}
\begin{aligned}
-up & = \frac{1}{2} p (h-\hat{h}) \v - \w p + p \v \hat{h} \\
      & = - \frac{1}{2} \v
      [\nabla \cdot (p\v)]^T - \w p + p\v \hat{h}
\end{aligned}
 \label{eqn:interm_up}
\end{equation}
where~\eqref{eqn:bvp_multi_Z_matrix} is used to obtain the second
equality.  Denoting $E:= \frac{1}{2} \v [\nabla \cdot (p\v)]^T$, a direct calculation shows that
\begin{equation*}
E_{l} +  \w_{l}p = \sum_{k=1}^d \frac{\partial }{\partial x_k}
\left( p[\v \v^T]_{lk} \right).
\end{equation*}

Substituting this in~\eqref{eqn:interm_up}, on taking the divergence
of both sides, we obtain
\begin{align}
-\nabla \cdot (pu) +\frac{1}{2}\sum_{l,k=1}^d \frac{\partial^2}{\partial x_l \partial
  x_{k}} \left( p [\v
 \v^T]_{lk} \right) &= \nabla \cdot (p\v) \hat{h}.
\label{eqn:FPK_23}
\end{align}
Using~\eqref{eqn:bvp_multi_Z_matrix} and~\eqref{eqn:FPK_23} in the forward equation~\eqref{eqn:mod_FPK},
\begin{align*}
\ud p & =  \clL^\dagger p + ( h-\hat{h} )^T
(\ud Z_t - \hat{h} \ud t)p.
\end{align*}
This is precisely the SDE \eqref{eqn:Kushner_eqn}, as desired.
\qed

\subsection{Proof of Theorem~\ref{thm:thm1}}
\label{apdx:uniqueness}

We omit the subscript ``$j$'' in this proof, writing just $\phi$ and
$h$.  We also suppress explicit dependence on time $t$, writing $p(x)$
instead of $p(x,t)$ and $\phi(x)$ instead of $\phi(x,t)$.  

Under Assumption~A1, $p$ is known to admit a spectral gap (or
Poincar\'{e} inequality) with constant $\lambda \ge
c_1$~\cite{Johnsen_2000}: That is, for all functions
$\phi\in H_0^1(\Re^d;p)$,
\begin{equation}
\int |\phi(x)|^2 p(x) \ud x \le \frac{1}{\lambda} \int |\nabla\phi(x)|^2 p(x)
\ud x.
\label{eqn:Poincare}
\end{equation}

Consider now the inner product
\[
<\phi,\psi>:=\int \nabla \phi(x) \cdot \nabla \psi(x) p(x) \ud x.
\]
On account of~\eqref{eqn:Poincare}, the norm defined by using the
inner product $<\cdot,\cdot>$ is equivalent to the standard norm in
$H^1(\Re^d;p)$.

\begin{romannum}
\item Consider the BVP in its weak form~\eqref{eqn:EL_phi_weak}.  The
  integral on the righthand-side is a bounded linear functional on
  $\psi \in H_0^1$, since
\begin{align*}
&|\int  (h(x)-\hat{h}) \psi(x) p(x) \ud x|^2  \nonumber\\
& \quad \le  \int
|h(x)-\hat{h}|^2 p(x) \ud x \int  |\psi(x)|^2 p(x) \ud x
\\
& \quad \le
(\text{const.}) \int  |\nabla \psi(x)|^2 p(x) \ud x,
\end{align*}
where~\eqref{eqn:Poincare} is used to obtain the second inequality.

It follows from the Reisz representation theorem that there exists a
unique $\phi \in H_0^1$ such that
\[
<\phi,\psi> =  \int  (h(x)-\hat{h}) \psi(x) p(x) \ud
x,
\]
for all $\psi \in H_0^1(\Re^d;p)$.  Thus $\phi$ is a weak solution
of the BVP, satisfying \eqref{eqn:EL_phi_weak}.

\item Suppose $\phi$ is a weak solution.  Using $\psi=\phi$
  in~\eqref{eqn:EL_phi_weak},
\begin{align*}
&\int |\nabla \phi|^2 p(x) \ud x  = \int  (h(x)-\hat{h}) \phi(x) p(x) \ud
x\\
& \le  \left( \int
|h(x)-\hat{h}|^2 p(x) \ud x\right)^{\frac{1}{2}} \left( \int  |\phi(x)|^2 p(x) \ud x \right)^{\frac{1}{2}}
\\ & \le
\left( \int
|h(x)-\hat{h}|^2 p(x) \ud x \right)^{\frac{1}{2}} \left( \frac{1}{\lambda} \int  |\nabla
\phi(x)|^2 p(x) \ud x\right)^{\frac{1}{2}}
\end{align*}
by \eqref{eqn:Poincare}. The estimate~\eqref{eqn:bound1} follows.

\item For the final estimate~\eqref{eqn:bound2}, we need:

\begin{lemma} \label{lem:secondderiv}
Under Assumptions A1-A4, the weak solution $\phi$ of the BVP \eqref{eqn:EL_phi_weak} belongs to $H^2(\Re^d;p)$, with
\begin{equation} \label{eqn:h2bound}
   \int |D^2 \phi|^2 p \ud x \leq  \int \nabla \phi \cdot G p \ud x
\end{equation}
where the vector function $G \in L^2(\Re^d;p)$ is defined by
\[
G = D^2(\log p) \nabla \phi + \nabla h
\]
and where $|D^2 \phi|^2= \sum_{j,k} (\frac{\partial^2 \phi}{\partial x_j \partial x_k})^2$.
\end{lemma}
\begin{proof}
First note that each entry of the Hessian matrix $D^2(\log p)$ is bounded, by Assumption A3, and that $\nabla h \in L^2(\Re^d;p)$ by Assumption A4. Hence $G \in L^2(\Re;p)$.

Next, elliptic regularity theory \cite[Section 6.3 of the PDE text by Evans]{EVANS}
%
%
%
applied to the weak solution $\phi \in H^1(\Re^d;p)$ says that $\phi
\in H^3_{loc}(\Re^d)$.
Hence the partial differential equation holds pointwise:
\begin{equation} \label{eqn:pointwise}
- \nabla \cdot (p \nabla \phi) = (h - \hat{h}) p .
\end{equation}
We differentiate with respect to $x_k$ to obtain:
\begin{align*}
& - \nabla \cdot (p \nabla \frac{\partial \phi}{\partial x_k}) - \nabla (\frac{\partial \log p}{\partial x_k}) \cdot (p\nabla \phi) - \frac{\partial \log p}{\partial x_k} \nabla \cdot (p \nabla \phi) \\
& = \frac{\partial h}{\partial x_k} p + (h-\hat{h}) \frac{\partial \log p}{\partial x_k} p .
\end{align*}
The final terms on the left and right sides cancel, by equation \eqref{eqn:pointwise}. Thus the preceding formula becomes
\begin{equation} \label{eqn:pointwisederiv}
- \nabla \cdot (p \nabla \frac{\partial \phi}{\partial x_k}) = G_k p ,
\end{equation}
where $G_k$ denotes the $k$th component of $G(x)$.

Let $\beta(x) \geq 0$ be a smooth, compactly supported ``bump'' function, meaning $\beta(x)$ is radially decreasing with $\beta(0)=1$. Let $s>0$ and multiply \eqref{eqn:pointwisederiv} by $\beta(sx)^2 \frac{\partial \phi}{\partial x_k}$. Integrate by parts on the left side (noting the boundary terms vanish because $\beta$ has compact support) to obtain 
\[
\int \nabla [\beta(sx)^2 \frac{\partial \phi}{\partial x_k}] \cdot (\nabla \frac{\partial \phi}{\partial x_k}) p \ud x = \int \beta(sx)^2 \frac{\partial \phi}{\partial x_k} G_k p \ud x .
\]
The right hand side $\text{RHS} \to \int \frac{\partial \phi}{\partial x_k} G_k p \ud x$ by dominated convergence as $s \to 0$, since $\beta(0)=1$. The left side,
\begin{align*}
\text{LHS}
& = \int \beta(sx)^2 |\nabla \frac{\partial \phi}{\partial x_k}|^2 p \ud x \\
& \quad + 2s \int \frac{\partial \phi}{\partial x_k} \beta(sx) (\nabla \beta)(sx) \cdot (\nabla \frac{\partial \phi}{\partial x_k}) p \ud x .
\end{align*}
Clearly the second term is bounded by
\begin{align*}
& 2s \lVert \nabla \beta \rVert_{L^\infty(\Re^d)} \int |\frac{\partial \phi}{\partial x_k}| \beta(sx) |\nabla \frac{\partial \phi}{\partial x_k}| p \ud x \\
& \leq s \lVert \nabla \beta \rVert_{L^\infty(\Re^d)} \int \left[(\frac{\partial \phi}{\partial x_k})^2 + \beta(sx)^2 |\nabla \frac{\partial \phi}{\partial x_k}|^2 \right] p \ud x
\end{align*}
and so
\begin{align*}
& (1-s \lVert \nabla \beta \rVert_{L^\infty(\Re^d)}) \int \beta(sx)^2 |\nabla \frac{\partial \phi}{\partial x_k}|^2 p \ud x \\
- & s \lVert \nabla \beta \rVert_{L^\infty(\Re^d)} \int (\frac{\partial \phi}{\partial x_k})^2 \ud x \leq \text{LHS} .
\end{align*}
Letting $s \to 0$ in LHS and RHS, and recalling that $\beta(x)$ is radially decreasing, we conclude from the monotone convergence theorem that
\[
\int |\nabla \frac{\partial \phi}{\partial x_k}|^2 p \ud x \leq \int \frac{\partial \phi}{\partial x_k} G_k p \ud x .
\]
Summing over $k$ completes the proof of the lemma.
\end{proof}

Next we prove \eqref{eqn:bound2}. We will use several times that
\begin{equation} \label{eqn:interim}
\int |\nabla \phi|^2 p \ud x \leq \frac{1}{\lambda} \int |h-\hat{h}|^2
p \ud x \leq \frac{1}{\lambda^2} \int |\nabla h|^2 p \ud x ,
\end{equation}
by \eqref{eqn:bound1} followed by \eqref{eqn:Poincare} applied to the function $h-\hat{h} \in H^1_0(\Re^d;p)$.

We have
\begin{align*}
\int |(D^2 \phi)(\nabla \phi)| p \ud x
\leq \int |D^2 \phi| |\nabla \phi| p \ud x \\
\leq (\frac{1}{\lambda^2} \int |\nabla h|^2 p \ud x)^{3/4} (\int |G|^2 p \ud x)^{1/4}
\end{align*}
by Lemma~\ref{lem:secondderiv}, Cauchy--Schwarz, and \eqref{eqn:interim}. The definition of $G$, the $L^2$-triangle inequality and \eqref{eqn:interim} show that
\begin{align*}
& (\int |G|^2 p \ud x)^{1/2} \\
& \leq \lVert D^2(\log p) \rVert_{L^\infty} (\int |\nabla \phi|^2 p \ud x)^{1/2} + (\int |\nabla h|^2 p \ud x)^{1/2} \\
& \leq \left( \frac{\lVert D^2(\log p) \rVert_{L^\infty}}{\lambda} + 1 \right) (\int |\nabla h|^2 p \ud x)^{1/2} .
\end{align*}
The estimate \eqref{eqn:bound2} now follows.
\end{romannum}





\bibliographystyle{plain}
\bibliography{strings,ACCPF,ACCJPDAFPF,CDC11,markov,refmtt,q}
\end{document}